\documentclass[a4paper,11pt]{article}

\usepackage{amsthm}
\usepackage{amsmath}
\usepackage{epsfig}
\usepackage{amssymb}
\usepackage{latexsym}
\usepackage[T1]{fontenc}
\usepackage[utf8]{inputenc}
\usepackage{color}
\usepackage{url}
\usepackage{tikz}
\usetikzlibrary {graphs, positioning}
\usetikzlibrary{graphs.standard}
\usetikzlibrary{fit}
\usetikzlibrary{arrows}
\usepackage[top=3cm, bottom=3cm, right=3cm, left=3cm]{geometry}

\author{
Edita Máčajová, Ján Mazák
\\[3mm]
\\{\tt \{macajova, mazak\}@dcs.fmph.uniba.sk}
\\[5mm]
Comenius University, Mlynská dolina, 842 48 Bratislava\\
}
\date{}
\title{On covering cubic graphs with three perfect matchings}

\theoremstyle{definition}

\newtheorem{theorem}{Theorem}

\newtheorem{lemma}[theorem]{Lemma}
\newtheorem{conjecture}{Conjecture}
\newtheorem{problem}{Problem}
\let\epsilon=\varepsilon

\let\phi=\varphi
\def\Q{\mathbb Q}

\begin{document}

\maketitle

\begin{abstract}
For a bridgeless cubic graph $G$, $m_3(G)$ is the ratio of the maximum number of edges of $G$ covered by the union of $3$ perfect matchings to $|E(G)|$. We prove that for any $r\in [4/5, 1)$, there exist infinitely many cubic graphs $G$ such that $m_3(G) = r$. For any $r\in [9/10, 1)$, there exist infinitely many cyclically $4$-connected cubic graphs $G$ with $m_3(G) = r$.
\end{abstract}

\section{Introduction}
For a bridgeless cubic graph $G$, let $m_3(G)$ be the ratio of the maximum number of edges of $G$ that can be covered by a union of three perfect matchings to $|E(G)|$. This problem was studied by Kaiser et al. in \cite{kaiser}: they proved that $m_3(G)\ge 27/35\approx 0.77$ and conjectured that the best possible lower bound is $4/5$, attained for the Petersen graph.

\begin{conjecture}[\cite{kaiser}]
For every bridgeless cubic graph $G$, $m_3(G)\ge 4/5$.
\label{c1}
\end{conjecture} 

Conjecture~\ref{c1} is a consequence of the Berge-Fulkerson conjecture \cite{patel2006unions} and implies Fan-Raspaud conjecture \cite{esperetmazzuoccolo}. It is also known that the problem of determining $m_3$ is NP-complete \cite{esperetmazzuoccolo}. This paper, building on the previous work \cite{AgarskyBP, 3pmManuscript}, investigates the set of all possible values of $m_3$. So far, it is not clear whether this set is an interval or not.

\begin{problem}
Let $r\in(0,1)\cap \Q$ be a fraction and $k\ge 2$ an integer. Does there exist a cyclically $k$-connected cubic graph $G$ such that $m_3(G) = r$?
\end{problem}

Judging from our work described in this article, it appears challenging to find graphs with a low value of $m_3$ among those with cyclic connectivity $4$ or even more. For cyclic connectivities 2 and 3, this task becomes easier because removing an edge or a vertex from the Petersen graph mostly preserves its structure, including a large fraction of edges that cannot be covered. Motivated by these observations, we suggest the following generalization of Conjecture~\ref{c1}.

\begin{problem}
For each $k\ge 2$, determine the largest constant $m_3^{(k)}$ such that every cyclically $k$-connected cubic graph $G$ different from the Petersen graph satisfies $m_3(G)\ge m_3^{(k)}$.
\end{problem}

Note that a graph $G$ has $m_3(G) = 1$ if and only if it is $3$-edge-colourable, so we are only interested in snarks (bridgeless cubic graphs with chromatic index $4$). A colour class in a $3$-edge-colouring of a subgraph is a matching, but it might not be a subset of any perfect matching of the whole graph, so $m_3$ is an invariant only very loosely related to resistance (the minimum number of edges that need to be removed from a cubic graph in order to obtain a $3$-edge-colourable graph). For instance, one can find two edges in the Petersen graph whose removal results in a $3$-edge-colourable subcubic graph with $13$ edges, but this graph will not have a cover consisting of three perfect matchings because the union of any three perfect matchings has at most $12$ edges.

A complete solution to the proposed problems is apparently very hard: for $k = 7$, we do not even know the answer for any single $r$, since no cyclically $7$-connected snarks are known (they are conjectured not to exist \cite{jaeger7conn}). As $k$ increases, the problem becomes more intriguing. In this paper, we provide partial answers for $k = 2$ and $k = 4$.

\begin{theorem}
For each fraction $p/q\in [4/5, 1)$, there exist infinitely many $2$-connected cubic graphs $G$ such that $m_3(G) = p/q$.
\label{thm:cc2}
\end{theorem}

\begin{theorem}
For each fraction $p/q\in [9/10, 1)$, there exist infinitely many cyclically $4$-connected cubic graphs $G$ (with girth $5$) such that $m_3(G) = p/q$.
\label{thm:cc4}
\end{theorem}

If Conjecture \ref{c1} is true, then our answer for $k = 2$ is complete. However, for $k = 4$, a gap remains; little is known about the interval $(4/5, 9/10)$. We believe that the exclusion of the Petersen graph would shift the lower bound on $m_3^{(4)}$ higher.

\begin{conjecture}
There exists a constant $c_4\in(4/5, 9/10]$ such that for every cyclically $4$-connected cubic graph $G$ different from the Petersen graph $m_3(G)\ge c_4$.
\label{conj:c4}
\end{conjecture}

On the other hand, a lower bound on $m_3^{(4)}$ cannot be moved all the way towards $9/10$. Using an exhaustive computer search, we discovered a snark $H$ on $28$ vertices with cyclic connectivity $4$ such that for any collection of three perfect matchings, there are at least $5$ uncovered edges, thus $m_3(H) = 37/42\approx 0.88$. No snark with up to $28$ vertices appears to provide a lower value of $m_3$ (almost all have a cover with only $3$ uncovered edges). The example of $H$ shows that if $c_4$ from Conjecture~\ref{conj:c4} exists, it is at most $37/42$.


We will describe our constructions of cubic graphs in terms of \emph{multipoles}, that is, cubic graphs with dangling edges. An edge of a multipole is a \emph{link} if it connects two vertices, a \emph{dangling edge} if only one of its ends is incident with a vertex, or an \emph{isolated edge} (multipoles with isolated edges are not used anywhere in this paper). The terminology on multipoles is fairly standard; details can be found in \cite{morphology}.

The notion of perfect matching can be straightforwardly extended to multipoles without isolated edges. A perfect matching of a multipole is a set of links and dangling edges such that each vertex is incident with exactly one of them.

\section{Cyclic connectivity 2}

Let $A$ and $B$ be the $2$-poles obtained from the Petersen graph and $K_4$, respectively, by cutting an edge into a pair of dangling edges. Join $a$ copies of $A$ and $b$ copies of $B$ in a circular fashion: one dangling edge of each of the multipoles is connected to the previous multipole and the other to the next one in the circular ordering (the copies of $A$ and $B$ can be arranged in an arbitrary order). Denote the resulting graph $G_{a,b}$. Since each copy of $A$ or $B$ is separated from the rest of the graph by a $2$-edge-cut, the cyclic connectivity of the resulting graph is $2$. We now prove that it has the properties required for our construction.

\begin{lemma}
Let $A$ be the $2$-pole obtained from the Petersen graph by cutting an edge into a pair of dangling edges. The following holds:
\begin{itemize}
    \item[(a)] If $M$ is a perfect matching of $A$, then it either contains both dangling edges or none of them.
    \item[(b)] If $M_1$, $M_2$, $M_3$ are perfect matchings of $A$ such that none of them contains a dangling edge, then there are at least $2$ links in $A$ not covered by any of $M_i$. For a suitable triple of matchings, it is possible to achieve equality.
    \item[(c)] If $M_1$, $M_2$, $M_3$ are perfect matchings of $A$ such that one of them contains a dangling edge, then there are at least $3$ links in $A$ not covered by any of $M_i$. For a suitable triple of matchings, it is possible to achieve equality.
\end{itemize}
\label{lemma:A2}
\end{lemma}

\begin{proof}
Part (a) is true because $A$ has an even number of vertices. 

Since the Petersen graph $P$ is edge-transitive, the choice of the edge to cut when creating $A$ does not affect the argument. A perfect matching of $A$ containing the dangling edges corresponds to a perfect matching of $P$ obtained by rejoining the dangling edges. In $P$, at most $12$ edges can be covered by a union of three perfect matchings, so at least $3$ edges will be uncovered (with equality achievable). The uncovered edges can either be three links, or two links and one link cut into a pair of dangling edges during the creation of $A$. This proves (b) and (c).
\end{proof}

\begin{lemma}
Let $B$ be the $2$-pole obtained from a $3$-edge-colourable cubic graph by cutting an edge into a pair of dangling edges. The following holds:
\begin{itemize}
    \item[(a)] If $M$ is a perfect matching of $B$, then it either contains both dangling edges or none of them.
    \item[(b)] There exist three perfect matchings $M_1$, $M_2$, $M_3$ of $B$ such that all the links and the dangling edges of $B$ are covered by them.
\end{itemize}
\label{lemma:B2}
\end{lemma}

\begin{proof}
Part (a) is true because $B$ has an even number of vertices. A triple of suitable matchings for part (b) are the colour classes of a $3$-edge-colouring of $B$.
\end{proof}

\begin{lemma}
For any integers $a\ge 1$ and $b\ge 0$,
$$
m_3(G_{a,b})=\frac{4a+2b}{5a+2b}.
$$
\label{lemma:fraction2}
\end{lemma}

\begin{proof}
Let us call links in the copies of $A$ and $B$ \emph{inner edges} and the edges arising from joining dangling edges \emph{outer edges}. There are $14$ and $5$ inner edges in each copy of $A$ and $B$, respectively. In addition, there are $a+b$ outer edges.  Altogether, this gives $15a+6b$ edges in $G_{a,b}$.

Consider a cover of $G_{a,b}$ by three perfect matchings $M_1$, $M_2$, $M_3$. If an outer edge is covered, then all outer edges are covered thanks to Lemmas~\ref{lemma:A2}(a) and \ref{lemma:B2}(a). Consequently, each copy of $A$ contains at least $3$ uncovered edges by Lemma~\ref{lemma:A2}(c). Otherwise, no outer edge is covered, and then we have at least $2$ uncovered edges in each copy of $A$ by Lemma~\ref{lemma:A2}(b) plus $a+b$ outer edges.

In either case, at least $3a$ edges are not covered, so at most $12a+6b$ are covered, hence
$$
m_3(G_{a,b}) \le {12a+6b\over 15a+6b} = {4a+2b\over 5a+2b}.
$$

The equality is easily achieved: we take a perfect matching $M_1$ containing all outer edges, and pick the rest of $M_1$ and both $M_2$ and $M_3$ according to Lemmas~\ref{lemma:A2}(c) and \ref{lemma:B2}(b).
\end{proof}

\begin{proof}[Proof of Theorem~\ref{thm:cc2}]
Consider the graph $G_{a,b}$ for $a = 2q-2p$, $b = 5p-4q$ (where $a>0$ because $p/q < 1$ and $b\ge 0$). According to Lemma \ref{lemma:fraction2},
$$
m_3(G_{a,b})={4(2q-2p)+2(5p-4q)\over 5(2q-2p)+2(5p-4q)} = {p\over q},
$$
so $G_{a,b}$ satisfies the required property. And so does $G_{ma, mb}$ for any positive integer $m$, thus there are infinitely many suitable graphs.
\end{proof}

\section{Cyclic connectivity 4}

The construction in the previous section is based on two ingredients. First, the key property of the multipole $A$ is its uncolourability, which ensures at least one uncovered edge in $A$. Second, addition of some colourable parts ``dilutes'' the effect of multipoles $A$, thus pushing $m_3$ upwards.

The fact that there are always three uncovered edges in $A$ (if we also count the dangling edges, each of them with weight $1/2$), and not just one, allows us to keep the lower bound of the interval $(4/5, 1)$ very low (optimal if Conjecture \ref{c1} is true). A similar method can be used for cyclic connectivity $k=3$: by removing a vertex from the Petersen graph, we still have an uncolourable multipole, which ensures at least one uncovered edge for every $15$ edges in the graph. The resulting ratio $23/27\approx 0.85$ is, however, rather far from $4/5$ \cite{AgarskyBP}. For $k\ge 4$, all the multipoles created from the Petersen graph would be colourable, and thus unsuitable for our construction.

Uncolourable multipoles that can be turned into snarks with cyclic connectivity $k$ are known for every $k\in\{4, 5, 6\}$: one can create them from snarks of large resistance (see \cite{KOCHOL2002281, Oddness, STEFFEN2004191}). Such multipoles can be used to construct a cyclically $k$-connected graph with $m_3$ equal to any fraction from the interval $(x, 1)$ for some $x$. It is, however, unclear how to find a multipole offering the best ratio of uncovered edges to size. One can employ a computer in search for best construction blocks \cite{AgarskyBP}, but the results are disappointing. The problem of finding all perfect matchings is computationally rather hard (both theoretically and in practice, even when one employs a SAT or an AllSAT solver). Moreover, larger snarks (or multipoles) tend to provide a lower proportion of uncovered edges compared to small ones. Here, we provide a construction that is verifiable by hand and results in a bound at least as good as anything we achieved with the help of a computer.

Consider the $(2,2)$-pole $A'$ depicted in Fig.~\ref{fig:A4}, composed of two copies $H_1$ and $H_2$ of the Blanuša block (obtained from the Petersen graph by removing two adjacent vertices \cite{blanusaJM}), $4$ additional vertices, and several additional edges. The dangling edges incident with $v_1$ and $v_6$ form the first connector, while the dangling edges incident with $v_{18}$ and $v_{19}$ form the second connector.

\begin{lemma}
A union of any three perfect matchings leaves uncovered at least $3$ links of $A'$ or at least $2$ links of $A'$ and $2$ dangling edges of $A'$. 
\label{lemma:a4}
\end{lemma}

\begin{proof}
Let $S_i$ be the set of edges of $A'$ (possibly dangling) covered by precisely $i$ perfect matchings. An edge from $S_2$ has exactly one neighbour from $S_0$ at each of its ends; an edge from $S_0$ has a neighbour from $S_2$ or $S_3$ at each of its ends.

It is a well-known property of the Blanuša block $H_1$ that the edges $e_3$ and $e_1$ must have the same colour in any $3$-edge-colouring; ditto for $e_3$ and $e_2$. Since $e_1$ is incident with $e_2$, $A'$ cannot be $3$-edge-colourable, and thus $E(A')\not\subseteq S_1$. Hence $S_3\cap S_2 \neq \emptyset$.

If a link $e\neq v_{18}v_{19}$ of $A'$ belongs to $S_3$, then $e$ has at least $3$ incident links in $A'$ that are uncovered (plus another link or a dangling edge), and if $e = v_{18}v_{19}\in S_3$, we have two uncovered links and two uncovered dangling edges incident to $e$, which proves the lemma. Otherwise, no link of $A'$ belongs to $S_3$, so each link of $A'$ belongs to at most two perfect matchings. If a dangling edge $e$ of $A'$ belongs to $S_3$, the two links incident with it are not covered, so each of them has a neighbour different from $e$ that is in $S_2$, and thus neighbours of neighbours are uncovered, hence we will also have at least $3$ uncovered edges in $A'$. We are left with the case $S_3 = \emptyset$ (so $S_2\neq \emptyset$). In this case, both $S_0$ and $S_2$ form a matching, and thus the subgraph $P_{02}$ induced by $S_0\cup S_2$ only has vertices of degree $2$.

If there is a cycle in $P_{02}$, it must be of length at least $6$, because it must be even and the girth of $A'$ is $5$. It thus contains at least $3$ uncovered edges. Otherwise, $P_{02}$ is a union of paths, each of the paths ending with a dangling edge on both ends. The shortest such path $P=v_1v_0v_4v_5v_6$ contains $5$ vertices, any other such path has at least $6$ vertices. But a path with $6$ vertices either contains $3$ uncovered links, or $2$ uncovered links and $2$ uncovered dangling edges.

We will prove that $P_{02} = P$ leads to a contradiction. We start by observing that $H_2$ has every edge covered exactly once (because $P_{02}\cap H_2 = \emptyset$). Thanks to the colouring properties of the Blanuša block \cite{blanusaJM}, edges $e_2$ and $e_3$ must belong to the same perfect matching, say, $M_1$. The other two matchings will be denoted by $M_2$ and $M_3$.

Let us denote $(v)$ the dangling edge incident with a vertex $v$. There are two possibilities for $P$, depending on whether $(v_1)$ is covered or not.

Case 1: $S_0 = \{v_1v_0, v_4v_5, (v_6)\}$. The edge $e_1$ belongs to both $M_2$ and $M_3$. Then $v_4v_3\in M_1$. Since the edges $v_7v_3$ and $v_7v_8$ are both incident with an edge from $M_1$, necessarily $v_7v_6\in M_1$. Then $v_6v_5\in M_2\cap M_3$, and so $v_5v_9\in M_1$. Look at the $5$-cycle $v_2v_3v_7v_8v_9$: none of its edges can belong to $M_1$, so it is covered by $M_2\cup M_3$. But that is obviously impossible for an odd cycle.

Case 2: $S_0 =\{(v_1), v_0v_4, v_5v_6\}$. Since $v_0v_1\in M_2\cap M_3$, necessarily $v_1v_2\in M_1$. Neither of $v_9v_2$ and $v_9v_8$ can be in $M_1$, so $v_9v_5\in M_1$. Then $v_5v_4\in M_2\cap M_3$, hence $v_4v_3\in M_1$. Again the $5$-cycle $v_2v_3v_7v_8v_9$ must be covered by $M_2\cup M_3$, a contradiction.
\end{proof}

\begin{figure}
    \centering
    \includegraphics{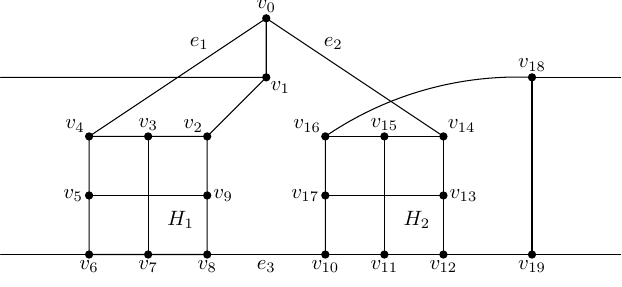}
    \caption{The multipole $A'$ used to construct cyclically $4$-connected graphs.}
    \label{fig:A4}
\end{figure}

For the ``diluting'' gadget, we take the $(2,2)$-pole $B'$ obtained from the Blanuša block by a suitable arrangement of its dangling edges. If the dangling edges are denoted by $f_1$, $f_2$, $f_3$, $f_4$ (viewed clockwise along the $8$-cycle), then the connectors will be $(f_1, f_4)$ and $(f_3, f_2)$.

Join $a$ copies of $A'$ ($a\ge 1$) and $b$ copies of $B'$ ($b\ge 0$) in a circular fashion; copies of $A'$ and $B'$ can be placed in an arbitrary order. Denote the resulting graph $G^4_{a,b}$. It has girth $5$ (easily verifiable) and cyclic connectivity $4$ (as sketched in the proof below). 

An I-extension is an operation that consists of inserting a vertex of degree $2$ into two edges of a graph and joining the added vertices by an edge. If we allow multigraphs, we can pick the same edge twice and then the I-extension would create a parallel edge.
Clearly, an I-extension performed on a cubic graph results in a cubic graph.

\begin{lemma}
Let $G'$ arise by I-extension from a cubic graph $G$ with cyclic connectivity $k$. The cyclic connectivity of $G'$ is either at least $k$ or equal to the length of the shortest cycle containing the edge $e$ added in the I-extension. Specifically, if $k\ge 4$ and $e$ creates neither a triangle nor a parallel edge, the cyclic connectivity of $G'$ is at least $4$.
\label{lemma:I}
\end{lemma}

\begin{proof}
If $e$ belongs to a cycle-separating cut $C'$ of $G'$, then $C'-\{e\}$ is a cycle-separating cut of size $|C'|-1$ in $G$, so $|C'|\ge k+1$. If $e$ belongs to a subgraph $H'$ separated by a cycle-separating cut $C'$ in $G'$, and the edges of $C'$ do not form a cycle-separating edge-cut in $G$, then the only possibility is that the addition of $e$ created a cycle in $H'$, while there was no cycle in the original subgraph $H$ in which we performed the $I$-extension (indeed: $H$ is separated from the rest of $G$ by the cut $C$ formed by the edges in $C'$; in case $C'$ contains an edge $e'$ that only arose during the I-extension, we put in $C$ the edge corresponding to $e'$ into which a vertex of degree $2$ was added when performing the I-extension). The number of edges leaving $H$ in $G$, i.e. $|C|$, is at least $|H|+2$ (since $H$ is cubic and acyclic), and the cycle created by adding $e$ has length at most $|H|+2$, so the cut $C'$ (with size equal to $|C|$) cannot push cyclic connectivity below the length of the newly added cycle containing $e$.
\end{proof}

Before calculating $m_3$ of the constructed graphs, we will explain why they are cyclically $4$-connected. Any copy of $A'$ arises by two I-extensions (adding the edges $v_{18}v_{19}$ and $v_0v_1$) from a simpler graph containing just two Blanuša blocks in place of $A'$. Each Blanuša block arises by $4$ subsequent I-extensions applied to two edges with no endvertex in common. A suitable sequence of I-extensions creates a copy of $A'$ from two edges (one corresponds to the dangling edges incident with $v_1$ and $v_{18}$, the other to the two remaining dangling edges). I-extensions do not decrease cyclic connectivity below $4$ in our case (no triangles or parallel edges, so Lemma~\ref{lemma:I} applies), and chains of copies of $B'$ (i.e. Blanuša blocks) are known to be cyclically $4$-connected (e.g. because they are part of generalized Blanuša snarks), which completes our explanation of why $G^4_{a,b}$ is cyclically $4$-connected. It is true also in case $b = 0$; we verified it for $G^4_{1,0}$ with a computer.

\begin{lemma}
For any integers $a\ge 1$ and $b\ge 0$,
$$
m_3(G^4_{a,b})={9a+4b\over 10a+4b}.
$$
\label{lemma:fraction4}
\end{lemma}

\begin{proof}
The graph $G = G^4_{a,b}$ has $30a + 12b$ edges. According to Lemma~\ref{lemma:a4}, at least $3a$ of them are uncovered in any union of three perfect matchings. Indeed, an uncovered link of $A'$ contributes an uncovered edge to $G$; an uncovered dangling edge corresponds to an uncovered edge in $G$ which is possibly counted twice if it connects two blocks $A'$, so we only counts its contribution as $1/2$. Hence
$$
m_3(G)\le {27a+12b\over 30a+12b}.
$$
We will prove the equality by describing three perfect matchings that cover $27a+12b$ edges of $G$. This collection of matchings can be visualised as a proper $3$-edge-colouring with specific defects (where colour classes correspond to the perfect matchings in the covering). In each copy of $A'$, the edges $v_{18}v_{19}$, $v_{12}v_{13}$, $v_{17}v_{16}$ are left uncovered and the edges $v_{12}v_{19}$, $v_{18}v_{16}$, $v_{17}v_{13}$ get pairs of colours 1 and 3, 2 and 3, 1 and 2, respectively. Each of the remaining edges of $G$ gets exactly one colour. 

There is a unique way of extending the colouring to the edges of $H_2$; in that colouring, both $v_0v_{14}$ and $(v_{18})$ get colour 1, while both $v_8v_{10}$ and $(v_{19})$ get colour 2. Next, we set the colours of $v_1v_2$, $v_0v_4$ and ($v_6)$ to 2 and the colour of $(v_1)$ to 1. Since the Blanuša block $H_1$ has all incoming edges coloured by the same colour 2, it is possible to colour all its links \cite{blanusaJM}. This colouring of $A'$ is compatible with a colouring of $B'$ which uses colour 1 for the edges $f_1$, $f_3$ and colour 2 for $f_2$, $f_4$ (such a colouring is known to exist \cite{blanusaJM}). It does not matter whether we joined $A'$ with $A'$, $A'$ with $B'$, or $B'$ with $B'$ when creating $G$---they all use the same pair of colours on the pairs of edges in the connectors used in the join.
\end{proof}

\begin{proof}[Proof of Theorem~\ref{thm:cc4}]
Consider the graph $G_{a,b}$ for $a = 4q-4p$, $b = 10p-9q$ ($a>0$ because $p/q < 1$). According to Lemma \ref{lemma:fraction4},
$$
m_3(G_{a,b})={9(4q-4p)+4(10p-9q)\over 10(4q-4p)+4(10p-9q)} = {p\over q},
$$
so $G^4_{a,b}$ satisfies the required property. And so does $G^4_{ma, mb}$ for any positive integer $m$, thus there are infinitely many suitable graphs.
\end{proof}

Our bound is better than the one mentioned by Agarsky \cite{AgarskyBP}, but this is only because his computation contains a mistake. However, he uses a gadget $A'$ with properties only verified by a computer, and it is not clear whether the code verifying it was correct (the version of his code we have access to contains a mistake: certain coverings resulting in $2$ uncovered links and $1$ uncovered dangling edge are ignored, which might affect his gadget $A'$).

\section{Acknowledgements}

This work was partially supported from the research grants APVV-19-0308, APVV-23-0076, VEGA 1/0743/21, VEGA 1/0727/22, and VEGA 1/0173/25.

\bibliographystyle{plain}
\bibliography{coverBy3PM}

\end{document}